\documentclass[11pt,a4paper]{article}
\usepackage[T1]{fontenc}
\usepackage{lmodern}
\usepackage[margin=28mm,headheight=14pt]{geometry}
\usepackage{amsmath,amssymb,amsthm,mathtools}
\usepackage{lmodern}
\usepackage{microtype,booktabs,longtable,array,enumitem}
\usepackage{fancyhdr}
\usepackage{xcolor}
\usepackage{tikz}
\usetikzlibrary{arrows.meta,calc}
\usepackage[hidelinks]{hyperref}
\usepackage[mathlines]{lineno}
\usepackage[nameinlink,noabbrev]{cleveref}
\numberwithin{equation}{section}
\newtheorem{theorem}{Theorem}[section]
\newtheorem{lemma}[theorem]{Lemma}
\newtheorem{proposition}[theorem]{Proposition}

\theoremstyle{definition}

\theoremstyle{remark}

\newcommand{\ii}{\mathrm i}
\newcommand{\dz}{\partial_z}
\newcommand{\db}{\partial_{\bar z}}
\newcommand{\R}{\mathbb R}
\newcommand{\C}{\mathbb C}
\newcommand{\D}{\mathbb D}
\newcommand{\J}{\mathsf J}
\newcommand{\Tr}{\operatorname{Tr}}
\newcommand{\Dom}{\operatorname{Dom}}
\newcommand{\Rea}{\operatorname{Re}}
\newcommand{\Ima}{\operatorname{Im}}

\definecolor{ink}{HTML}{243441}

\begin{document}
\title{On the Dirac Faber-Krahn Conjecture}
\author{Habib Ammari\footnote{
  ETH Z\"urich, Department of Mathematics, Rämistrasse 101, 8092 Z\"urich, Switzerland, (habib.ammari@math.ethz.ch).} \and 
  Jiayu Qiu\footnote{Corresponding author. 
  ETH Z\"urich, Department of Mathematics, Rämistrasse 101, 8092 Z\"urich, Switzerland, (jiayu.qiu@math.ethz.ch).}}
\date{}
\maketitle
\thispagestyle{empty}
\begin{abstract}
In this paper, we prove the strict Faber-Krahn inequality for non-negative-mass Dirac operators on simply connected domains with infinite-mass boundary condition. Denoting the first positive eigenvalue of the corresponding Dirac operator with mass $m\geq 0$ as $\lambda_1^{+}(\Omega;m)$, our main result shows that
\begin{equation*}
\lambda_1^{+}(\Omega;m)\geq \lambda_1^{+}(B_{|\Omega|};m),
\end{equation*}
where $B_{|\Omega|}$ is the disk with the same area as the domain $\Omega$, and the equality is strictly valid if and only if $\Omega=B_{|\Omega|}$. The key observation leading to the proof is that the upper component of the Dirac spinor associated with $\lambda_1^{+}(\Omega;m)$ is non-vanishing over $\Omega$, which is proved by exploiting its Dirac current and the associated stream function. Based on this global non-vanishing property, we construct a sphere-valued spinor map of degree one, which allows a sharp comparison with the radial disk profile using the planar isoperimetry.
\end{abstract}

\vspace{1em}


\begingroup
\small
\setlength{\parskip}{0pt}
\tableofcontents
\endgroup

\section{Introduction and main result} \label{sec:intro}

Let $\Omega\subset\R^2$ be a bounded simply connected domain with
$C^\infty$ boundary, and denote its area (the Lebesgue measure) by $A=|\Omega|$. The origin-centered disk with area $A$ is denoted by $B_A:=\{x\in\R^2:|x|<\sqrt{A/\pi}\}$. For the fixed non-negative mass $m\geq0$, we consider the Dirac operator
\begin{equation} \label{intro:operator}
D_{\Omega,m}=-\ii(\sigma_1\partial_x+\sigma_2\partial_y)+m\sigma_3
\end{equation}
with the infinite-mass boundary condition (IMB)
\begin{equation*}
v=\ii\nu u\quad\text{on }\partial\Omega .
\end{equation*}
Here, $u$ and $v$ denote the (complex) components of the Dirac spinor, $\sigma_i$ denote the Pauli matrices
\begin{equation}\label{pre:pauli}
 \sigma_1=\begin{pmatrix}0&1\\1&0\end{pmatrix},\quad
 \sigma_2=\begin{pmatrix}0&-\ii\\\ii&0\end{pmatrix},\quad
 \sigma_3=\begin{pmatrix}1&0\\0&-1\end{pmatrix},
\end{equation}
and $\nu=\nu_1+\ii\nu_2$ is the outward unit normal identified as a complex number. The self-adjoint realization and general characterization of the spectrum of $D_{\Omega,m}$ are presented in Section \ref{sec:general_spectrum}. In particular, arranging the eigenvalues of $D_{\Omega,m}$ in increasing order and denoting the first positive one as $\lambda_1^+(\Omega;m)$, our main result is the following.
\begin{theorem}[Strict Faber-Krahn inequality]\label{main:fk}
For every bounded $C^\infty$ simply connected domain
$\Omega\subset\R^2$ and $m\geq0$, the first positive eigenvalue $\lambda_1^+(\Omega;m)$ is simple, with the associated eigenspinor $(u,v)^\top$ satisfying $u(z)\ne0$ for all $z\in\overline\Omega$. Moreover, it holds that
\begin{equation}\label{main:ineq}
 \lambda_1^+(\Omega;m)\geq\lambda_1^+(B_A;m),
 \quad A=|\Omega|.
\end{equation}
The equality holds if and only if $\Omega=B_A$ up to translation.
\end{theorem}

The classical Faber–Krahn inequality \cite{Faber,Krahn} states that, among planar domains of prescribed area, the disk uniquely minimizes the first Dirichlet eigenvalue of the Laplacian. The corresponding question for the Dirac operator concerns its first positive eigenvalue under the infinite-mass boundary condition. In the massless case, the following area bound is obtained in \cite{BFSVGap}
\[
 \lambda_1^+(\Omega;0)^2\geq\frac{2\pi}{|\Omega|}.
\]

This estimate also follows from the two-dimensional case of Raulot's inequality for the Dirac operator with chiral boundary conditions \cite[Theorem 1]{Raulot}. In the Riemannian setting, equality is attained precisely by a round hemisphere. For planar domains, the bound is strict even for the disk, leaving open the determination of the sharp planar constant. Later, \cite{LO} establishes a sharp upper bound. In \cite[Conjecture 1]{ABLO}, the authors formulate a variational characterization of the first positive eigenvalue on simply connected smooth domains and propose the Dirac Faber-Krahn conjecture on a sharp lower bound; they have also provided numerical evidence supporting the conjecture. 

The difficulty of the proof persists even within elementary classes of domains. In \cite{BK}, the authors prove partial optimization results for rectangles, and \cite{VuTriangles} obtains comparisons for right triangles under additional geometric hypotheses. Subsequently, the authors of \cite{ABK} study the optimization of the first three positive eigenvalues numerically, for different masses and geometric constraints. Their results support the disk conjecture for the first positive eigenvalue and also indicate that higher eigenvalues can exhibit different optimizing geometries. More recently, the connection between the Dirac Faber-Krahn conjecture and the $\bar\partial$-Robin Laplacian has led to further progress. Following \cite{DuranRobin}, the authors of \cite{DMS} apply this tool and prove Faber-Krahn inequalities for boundary parameters sufficiently close to the zigzag endpoints. Later, \cite{DuranBounds} generalizes the result to thin convex domains; see the more recent result \cite{DMSPW}. When the mass term is sufficiently large, the non-relativistic limit studied in \cite{BHS} yields disk comparison for a fixed domain after the corresponding rescaling. 

These results address different geometric restrictions, parameter regimes, or directions of comparison. The main result of this work, Theorem \ref{main:fk}, addresses the open problem on the strict Faber-Krahn-type inequality. We work with the fixed infinite-mass boundary condition and do not restrict the mass term in the Dirac operator. We point out that, to the best of our knowledge, the simplicity of the first positive eigenvalue and the non-vanishing of the upper eigenspinor component claimed in Theorem \ref{main:fk} are also new; in fact, despite of their independent interests, these two key observations serve as the main ingredient for the proof of the Faber-Krahn inequality. These new structural conclusions provide a replacement for part of the role played by a positive ground state in scalar elliptic problems. In particular, they allow the first positive spectral branch to be identified and continued under smooth conformal deformations, which establishes a bridge to compare the first positive eigenvalue of a general domain to that of the disk.

The principal analytic tool is the stream function of the conserved Dirac current. The construction of such potentials from spinor bilinears has classical precedents in the spinor representation of surfaces\footnote{In fact, after rescaling the eigenspinor, our stream function is proportional to a height coordinate in the generalized Weierstrass representation.}; see \cite{Friedrich,Taimanov,Taimanov2}. The contribution here is its use together with the infinite-mass boundary condition: its boundary normalization leads to an exact weighted scalar identity, positivity estimates, and a criterion that identifies the first positive eigenspinor. A local description of boundary zeros then makes these properties stable under continuation from a disk. To our knowledge, this combination has not previously been used in the Dirac Faber-Krahn problem.

It is worthwhile to mention the recent breakthrough in the overdetermined spectral problems. In \cite{ColbrookStepaniants}, the authors  construct a simply connected noncircular domain with analytic boundary that is a counterexample to the planar Schiffer conjectures. Their domain supports a non-constant Helmholtz solution with constant boundary value and vanishing normal derivative. In \cite{CaoDeDios}, the authors independently constructed another counterexample using a bifurcation method. These results show why overdetermined Helmholtz data alone do not provide a general route to prove the disk rigidity. These achievements indeed inspire this work: we realize the non-vanishing property of the eigenspinor should provide an essential additional ingredient to the conjecture.

The argument suggests several further questions. First of all, for other local quantum-dot boundary conditions, it will be interesting whether our method of stream potential can yield a sharp geometric comparison. On the other hand, for domains with holes that are not covered in this work, we point out that the stream function may take different constants on different components; understanding these constants is a natural step towards inequalities for domains with nontrivial topology. These extensions, as well as the passage to nonsmooth domains, require additional arguments and are left for future study.


The proof of Theorem~\ref{main:fk} consists of two main parts. First, in Sections \ref{stm:section}-\ref{con:section}, we prove that the upper component of the first positive eigenspinor is non-vanishing on $\overline\Omega$. This property also implies simplicity of the eigenvalue. These results are summarized in Theorem \ref{con:nonvanishing}. The stream function of the Dirac current is mainly constructed to prove this part. Second, we use this non-vanishing component to construct a sphere-valued map and obtain a sharp geometric comparison with the disk, which is accomplished in Section \ref{sec:application}. The major geometric result used in Section \ref{sec:application} is essentially a classical level-set comparison method, with details presented in Section \ref{sec:cap}.
\section{Notation and preliminaries}\label{sec:prelim}

\subsection{Coordinates, orientations, and function spaces}
We identify $\R^2$ with $\C$ by $z=x+\ii y$. The unit complex disk is denoted as $\D=\{z\in\C:|z|<1\}$. For any complex number (or vector) $a$, $Kc:=\overline{c}$ denotes its (entry-wise) complex conjugation. For vectors in $\R^2$, the $\frac{\pi}{2}$-rotation is denoted as  $\J(x,y):=(-y,x)$, or multiplication by $\ii$ in complex notation. The Wirtinger derivatives on complex functions, as well as the Laplacian, are denoted as
\begin{equation}\label{pre:wirtinger}
 \dz:=\tfrac12(\partial_x-\ii\partial_y),\quad
 \db:=\tfrac12(\partial_x+\ii\partial_y),\quad
 \Delta=\partial_x^2+\partial_y^2:=4\dz\db.
\end{equation}
In particular, for a real function $p$, $2\db p=p_x+\ii p_y$ identifies its gradient with a complex-valued function. Throughout the paper, a superscript $\top$ denotes the transpose, where the superscript $*$ denotes the conjugate transpose.

On the positively oriented boundary $\partial \Omega$, the unit tangent with counterclockwise orientation is denoted as $\tau:=\J\nu=\ii\nu$. The normal derivative is written as $\partial_\nu f:=\nu_1f_x+\nu_2f_y$, and $\partial_\tau f:=\tau_1f_x+\tau_2f_y$ is the tangential derivative.

Throughout this paper, $L^p(\Omega;\C^q)$ refers to the $p$-integrable complex vector-valued function space with the usual norm induced by the inner product
\[
 \langle f,g\rangle_{L^2}:=\int_\Omega\sum_{j=1}^q
                         f_j\overline{g_j}.
\]
For $k,p\geq 0$, the Sobolev space $W^{k,p}(\Omega)$ of scalar functions is defined in the standard way, the derivative being understood in the distributional sense; when $p=2$, we write $H^k(\Omega):=W^{k,2}(\Omega)$. Since $\Omega$ is assumed to have a smooth boundary, the trace map $\Tr:H^{s}(\Omega)\to H^{s-1/2}(\partial\Omega)$ is well-defined, where the later denotes the standard fractional boundary Sobolev space defined by the Fourier transform.

\subsection{Dirac operator and general properties of its spectrum} \label{sec:general_spectrum}

Here, we sketch some general properties of the Dirac operator \eqref{intro:operator}, or equivalently the following compressed form
\begin{equation} \label{pre:matrix}
D_{\Omega,m}=
 \begin{pmatrix}m&-2\ii\dz\\-2\ii\db&-m\end{pmatrix}.
\end{equation}
First, as proved in \cite{BFSVSA}, $D_{\Omega,m}$ is self-adjoint in $L^2(\Omega;\C^2)$ with domain
\begin{equation} \label{pre:domain}
\Dom(D_{\Omega,m})=
 \{(u,v)^\top\in H^1(\Omega;\C^2):\Tr v=\ii\nu\Tr u\}.
\end{equation}
In terms of the spinor $\psi=(u,v)^{\top}$, the boundary condition is equivalent to $(\sigma\cdot\tau)\psi=\psi$, where $\sigma\cdot\tau=\sigma_1\tau_1+\sigma_2\tau_2$. It is also proved in \cite{BFSVSA} that $D_{\Omega,m}$ has a compact resolvent (recall that $\Omega$ is bounded). Consequently, the spectrum $\text{Spec}(D_{\Omega,m})$ is purely discrete. Moreover, since $D_{\Omega,m}$ respects the particle-hole symmetry
\begin{equation} \label{eq:particle-hold}
\mathcal{C}D_{\Omega,m}\mathcal{C}^{-1}=-D_{\Omega,m} \quad \text{with}\quad \mathcal{C}:=\sigma_{1}K \quad \text{satisfying} \quad  \mathcal{C}^2=I,
\end{equation}
we know $\text{Spec}(D_{\Omega,m})$ is symmetric about the origin. Hence, ordering the eigenvalues increasingly, it holds that
\begin{equation*}
\text{Spec}(D_{\Omega,m})=\big\{\cdots\leq \lambda_2^{-}(\Omega;m)\leq \lambda_1^{-}(\Omega;m) \leq 0 \leq \lambda_1^{+}(\Omega;m)
\leq \lambda_2^{+}(\Omega;m)\leq \cdots \big\},
\end{equation*}
where
\begin{equation*}
\lambda_n^{+}(\Omega;m)=-\lambda_n^{-}(\Omega;m) .
\end{equation*}
As proved in \cite[Theorem 3]{Vu}, each eigenspinor of $D_{\Omega,m}$ lies in $C^{\infty}(\overline{\Omega})$. Finally, in the following Lemma, we show that $\text{Spec}(D_{\Omega,m})$ has a gap near zero. Suppose $\lambda\in \text{Spec}(D_{\Omega,m})$ with an eigenspinor $\psi=(u,v)^{\top}$. In the sequel, we will write
\begin{equation}\label{pre:ab}
 a:=\lambda+m,\quad b:=\lambda-m,
 \quad U:=|u|^2,\quad V:=|v|^2.
\end{equation}
Then the scalar components satisfy
\begin{equation}\label{pre:system}
 \db u=\frac{\ii a}{2}v,\quad
 \dz v=\frac{\ii b}{2}u,
 \quad v=\ii\nu u\text{ on }\partial\Omega.
\end{equation}

\begin{lemma}\label{pre:gap}
The densities $U$ and $V$ in \eqref{pre:ab} satisfy that
\begin{equation}\label{pre:balance}
 \int_{\partial\Omega}Uds
       =b\int_\Omega UdA-a\int_\Omega VdA.
\end{equation}
Consequently, $0\notin \text{Spec}(D_{\Omega,m})$, and every positive eigenvalue satisfies $\lambda>m$. 
\end{lemma}
\begin{proof} 
First, recall that for a complex-valued smooth function $F$, the divergence theorem applied to its real and imaginary parts gives
\begin{equation}\label{pre:green}
 \int_\Omega 2\dz F\,dA=\int_{\partial\Omega}\bar\nu F\,ds,
 \quad
 \int_\Omega 2\db F\,dA=\int_{\partial\Omega}\nu F\,ds.
\end{equation}
On the other hand, by \eqref{pre:system}, it is direct to calculate that
\begin{equation} \label{eq:dirac_current_derivative}
2\dz(\bar uv)=
 2\bigl(\overline{\db u}\,v+\bar u\dz v\bigr)
                  =\ii(bU-aV).
\end{equation}
Note that the boundary value of $\bar\nu\bar uv$ is $\ii U$. Hence, equation~\eqref{pre:green}, together with the above identity, proves \eqref{pre:balance}. 

The claim of the spectral gap is justified case by case:
\begin{enumerate}
    \item[i)] if $0<\lambda<m$, then we have $b<0<a$ and the right side of \eqref{pre:balance} is strictly negative for a nontrivial eigenspinor, which contradicts the positivity of the left side;
    \item[ii)] if $\lambda=m>0$, the right side of \eqref{pre:balance} equals $-2m\int V$. Hence $v=0$ and the boundary trace of $u$ is zero. However, then the analyticity $\db u=0$ and the zero trace of $u$ force $u\equiv 0$, which is again a contradiction;
    \item[iii)] if $0=\lambda<m$, the right side of \eqref{pre:balance} equals $-m\int(U+V)<0$, which is impossible as in i);
    \item[iv)] if $\lambda=m=0$, \eqref{pre:balance} indicates that $u$ vanishes identically on the boundary, and so does $v$ by IMB. Since $u$ and $v$ are holomorphic and antiholomorphic respectively, as seen from \eqref{pre:system}, they are all harmonic; hence their zero traces imply global vanishing, which is again a contradiction.
\end{enumerate}
\end{proof}
In the remainder of this paper, for every positive eigenvalue $\lambda$, we therefore know that
\begin{equation}\label{pre:k}
 k:=\sqrt{ab}=\sqrt{\lambda^2-m^2}>0,
\end{equation}
and the components of an eigenspinor satisfy that Helmholtz equation
\begin{equation} \label{eq:helmholtz_component}
-\Delta u=k^2u,\quad -\Delta v=k^2v .
\end{equation}
In particular, both components $u$ and $v$ are real analytic in the interior of $\Omega$ by the regularity theory of elliptic equations.

Note that when the concerned domain is actually a disk, i.e., $\Omega=B_{|\Omega|}$ with radius $R$, the eigenspinor associated with $\lambda=\lambda_{1}^{+}$ is explicitly calculated in the polar coordinates as
\begin{equation*}
u(r,\theta)=J_0(kr),\quad v(r,\theta)=i\,\frac{k}{\lambda+m}e^{i\theta}J_1(kr)
\end{equation*}
when \(k<j_{0,1}/R\) satisfies the compatibility condition 
\begin{equation*}
J_0(kR)=\frac{k}{\sqrt{m^2+k^2}+m}J_1(kR).
\end{equation*}
See Section \ref{sec:disk_first_eigenspinor} for more details. Here $J_0$ and $J_1$ are the Bessel functions of the first kind of orders zero and one, respectively,
\begin{equation}\label{pre:besselseries}
 J_0(x)=\sum_{j=0}^\infty\frac{(-1)^j(x/2)^{2j}}{(j!)^2},
 \quad
 J_1(x)=\sum_{j=0}^\infty\frac{(-1)^j(x/2)^{2j+1}}{j!(j+1)!},
\end{equation}
and $j_{0,1}$ is the first positive (simple) zero of $J_0$. We record the following elementary property regarding the Bessel functions that will be applied later \cite{DLMF}. The proof is omitted.
\begin{lemma}\label{pre:bessel}
On $0<x<j_{0,1}$ the function $h(x)=J_1(x)/J_0(x)$ is positive
and increases strictly from zero to infinity. It satisfies
\begin{equation}\label{pre:riccati}
 h'(x)=1+h(x)^2-\frac{h(x)}x,
 \quad h(x)=\frac x2+O(x^3)\quad(x\to 0^{+}).
\end{equation}
\end{lemma}
\section{The stream function}\label{stm:section}

Suppose $\lambda\in \text{Spec}(D_{\Omega,m})$ with an eigenspinor $\psi=(u,v)^{\top}$. Recall the quantities $a,b,U,V$ associated with this eigenpair defined in \eqref{pre:ab}. In this section, we study the Dirac current \(j=2\bigl(\operatorname{Re}(\bar uv),\operatorname{Im}(\bar uv)\bigr)\) and its associated stream function $p$, constructed in Proposition \ref{stm:stream}, that satisfies
\begin{equation*}
\nabla p=\tfrac12 \J j .
\end{equation*}
In particular, the stream function characterizes the energy cost of spatially modulating an eigenspinor’s upper component in the reduced scalar problem; see Lemma \ref{stm:transform}. The main result of this section is Lemma \ref{stm:zerofree}, which shows that 1) if the upper component of an eigenspinor is non-vanishing, then the stream function is non-negative throughout the domain, and 2) using 1) and Lemma \ref{stm:transform}, we show that the eigenspinor with non-vanishing upper component, if exists, belongs to the first positive eigenvalue, which is necessarily simple.

\subsection{Construction of the real stream function}

\begin{proposition}[Stream function] \label{stm:stream}
There is a unique real smooth function $p$ on $\overline\Omega$
that satisfies
\begin{equation}\label{stm:streamgradient}
 p|_{\partial\Omega}=0,\quad
 p_x+\ii p_y=\ii\bar u v.
\end{equation}
It also satisfies 
\begin{equation}\label{stm:streamidentities}
 -\Delta p=bU-aV,\quad
 |\nabla p|^2=UV,\quad
 \text{and}\quad 
 p_\nu=-U\quad\text{on }\partial\Omega.
\end{equation}
In the sequel, $p$ is referred to as the stream function (associated with the eigenspinor $\psi$).
\end{proposition}
\begin{proof}
Note that any real vector field $X=(X_1,X_2)$ satisfies
\[
 2\dz(X_1+\ii X_2)
 =\operatorname{div}X+
             \ii(\partial_xX_2-\partial_yX_1).
\]
Hence, \eqref{eq:dirac_current_derivative} indicates that the real vector field identified with the complex scalar field $\ii\bar u v$ is curl-free, which then admits a potential antiderivative whose gradient is equal to $\ii\bar u v$, due to the simple connectedness of $\Omega$; denoting this potential as $p$, the second identity in \eqref{stm:streamgradient} is proved. On the other hand, by the IMB,
\begin{equation*}
\ii\bar uv=\ii\bar u(\ii\nu u)=-U\nu\quad \text{on $\partial \Omega$}.
\end{equation*}
This implies that the gradient of $p$ has zero tangential component along $\partial \Omega$; therefore, $p$ is constant along the boundary. Shifting $p$ by its boundary if necessary, we conclude the proof of \eqref{stm:streamgradient}. Finally, the remaining identities in \eqref{stm:streamidentities} are checked by direct calculation, which are omitted here.
\end{proof}

\subsection{The scalar equation of upper component and its weak form}
In this section, we focus on the upper component $u$ of the eigenspinor. In particular, we show that the stream function constructed in Proposition \ref{stm:stream} characterizes the energy cost of spatially modulating an eigenspinor’s upper component in the reduced scalar problem. To start, we note that by \eqref{pre:system},
\begin{equation}\label{stm:scalar}
 -\Delta u=k^2u\quad\hbox{in }\Omega,\quad
 2\bar\nu\db u+au=0\quad\hbox{on }\partial\Omega.
\end{equation}
The weak formulation of \eqref{stm:scalar} is defined via the sesquilinear form:
\begin{equation} \label{stm:form}
t_{\lambda,m}[w,z]
:=4\int_\Omega\db w\,\overline{\db z}\,dA +a\int_{\partial\Omega}(\Tr w)\overline{\Tr z}\,ds-k^2\int_\Omega w\bar z\,dA,
\end{equation}
for $w,z\in H^1(\Omega;\C)$; we also write $t_{\lambda,m}[w]:=t_{\lambda,m}[w,w]$. Hence, for $u$ solving \eqref{stm:scalar}, it satisfies
\begin{equation}\label{stm:weak}
 t_{\lambda,m}[u,z]=0 \quad(\forall z\in H^1(\Omega;\C)).
\end{equation}
The main result of this subsection is the following:

\begin{lemma} \label{stm:transform}
Set $W:=U-ap$. For every $\xi\in H^1(\Omega;\C)$,
\begin{equation}\label{stm:ground}
 t_{\lambda,m}[u\xi]
       =4\int_\Omega W|\db\xi|^2\,dA
                       +4a\int_\Omega p|\dz\xi|^2\,dA.
\end{equation}
\end{lemma}
\begin{proof}
First, we take $\xi\in C^\infty(\overline\Omega)$. Note that the product rule yields
\[
 |\db(u\xi)|^2
 =|\xi|^2|\db u|^2+U|\db\xi|^2
       +2\operatorname{Re}(\xi\,\bar u\,\db u\,
                                      \overline{\db\xi}).
\]
Also,
\[
 \overline{\db(u|\xi|^2)}
 =|\xi|^2\overline{\db u}
       +\bar u\bigl(\xi\,\overline{\db\xi}+\bar\xi\,\dz\xi\bigr).
\]
Using these identities, when subtracting $\operatorname{Re}t_{\lambda,m}[u,u|\xi|^2]=0$ from $t_{\lambda,m}[u\xi]$, one sees that the zeroth-order volume and boundary terms, as well as the $|\xi|^2|\db u|^2$ terms, all cancel out exactly. Moreover, using the following identity derived from  \eqref{pre:system} and \eqref{stm:streamgradient}
\begin{equation}\label{stm:productidentity}
 \bar u\,\db u=a\,\db p,
\end{equation}
one can check that the remaining terms in $t_{\lambda,m}[u\xi]-\operatorname{Re}t_{\lambda,m}[u,u|\xi|^2]$ are
\begin{equation}\label{stm:remainder}
t_{\lambda,m}[u\xi]-\operatorname{Re}t_{\lambda,m}[u,u|\xi|^2]=
 4\int_\Omega U|\db\xi|^2+
 4a\operatorname{Re}\int_\Omega
       \db p\,\bigl(\xi\,\dz\bar\xi-\bar\xi\,\dz\xi\bigr).
\end{equation}
Let $A_\xi:=\xi\,\dz\bar\xi-\bar\xi\,\dz\xi$. Its derivative is
\[
 \db A_\xi
 =|\db\xi|^2-|\dz\xi|^2
                  +\frac14(\xi\,\Delta\bar\xi-\bar\xi\,\Delta\xi).
\]
with the last term being purely imaginary. Thus, recalling that $p$ is real and vanishes on the boundary, the second formula in
\eqref{pre:green} gives
\[
 \operatorname{Re}\int_\Omega\db p\,(\xi\,\dz\bar\xi-\bar\xi\,\dz\xi)
 =-\operatorname{Re}\int_\Omega p\,\db A_\xi
 =\int_\Omega p\bigl(|\dz\xi|^2-|\db\xi|^2\bigr).
\]
Substituting it into \eqref{stm:remainder} proves \eqref{stm:ground} for $\chi\in C^{\infty}(\overline{\Omega})$. Finally, for general $\xi\in H^1$, approximating it by smooth functions together with the $H^1$ boundedness of $t_{\lambda,m}$ (see, for instance, \cite{Hunter}) concludes the proof.
\end{proof}

\subsection{A nonvanishing upper component determines the first mode}
Now we prove the main result of this section:

\begin{lemma}\label{stm:zerofree}
Suppose that $\psi=(u,v)^{\top}$ is an eigenspinor associated with a positive eigenvalue $\lambda$ and $u$ is nonvanishing on
$\overline\Omega$. Then, it holds that
\begin{equation}\label{stm:positiveweights}
 p>0\quad\hbox{in }\Omega,\quad
 W=U-ap\geq\min_{\partial\Omega}U>0
                   \quad\hbox{on }\overline\Omega.
\end{equation}
Moreover, $\lambda=\lambda_1^{+}$ and must be simple.
\end{lemma}
\begin{proof}
We claim that $p$ does not have local minimum inside $\Omega$; then the vanishing trace of $p$ concludes that $p>0$ in $\Omega$. Indeed, at an interior critical point of $p$, the identity $|\nabla p|^2=UV$ and the fact that $U>0$ (recall that $u$ is nonvanishing) force $v=0$. Hence, $\Delta p=-bU<0$ at the critical point, which prevents it from being a local minimum.

For the second inequality in \eqref{stm:positiveweights}, note that \eqref{stm:productidentity} gives $\db W=u\,\overline{\dz u}$. Moreover,
\[
 \Delta U
 =2|\nabla u|^2-2k^2U
 =4|\dz u|^2+a^2V-2k^2U,
 \quad
 \Delta W=4|\dz u|^2-k^2U.
\]
At a critical point of the real function $W$, its $\db$ derivative
vanishes. Since $u\ne0$, the identity $\db W=u\,\overline{\dz u}$ forces $\dz u=0$, and hence $\Delta W=-k^2U<0$. Similarly as in the first part, we then see that $W$ does not have an interior local minimum, and hence \eqref{stm:positiveweights} is proved.

Next, we prove $\lambda=\lambda_1^{+}$. Note that, since $u$ is non-vanishing and smooth, multiplication by $u$ is an isomorphism of $H^1(\Omega)$. Therefore, \eqref{stm:ground} and
\eqref{stm:positiveweights} show that $t_{\lambda,m}[w]\geq0$ for every $w\in H^1(\Omega)$. On the contrary, suppose that there is another positive eigenvalue $\widetilde\lambda<\lambda$ and $(\widetilde u,\widetilde v)^\top$ is an associated eigenspinor. Then, by the definition of $t_{\lambda,m}$ and \eqref{pre:balance}, one calculates that
\begin{equation*}
\begin{aligned}
t_{\lambda,m}[\widetilde u]
 &=(\lambda-\widetilde\lambda)
 \left[\int_{\partial\Omega}|\widetilde u|^2
              -(\lambda+\widetilde\lambda)\int_\Omega|\widetilde u|^2
 \right] \\
 &=-(\lambda-\widetilde\lambda)
 \left[(\lambda+m)\|\widetilde u\|_2^2
        +(\widetilde\lambda+m)\|\widetilde v\|_2^2\right]<0,
\end{aligned}
\end{equation*}
which is a contradiction. This proves $\lambda=\lambda_1^{+}$.

Finally, suppose that $\lambda$ is not simple and $(\widehat u,\widehat v)^{\top}$ is another associated eigenspinor. The upper component $\widehat u$ satisfies $t_{\lambda,m}[\widehat u]=0$ and can be written $\widehat u=u\xi$, with $\xi\in H^1$. Then, the identity \eqref{stm:ground} and $W>0$ imply $\db\xi=\dz\xi=0$ almost everywhere in $\Omega$; hence $\xi$ is constant. Using \eqref{pre:system}, one can also prove that $\widehat v=\xi v$. This proves the simplicity of $\lambda$.
\end{proof}

\section{Boundary zeros and sign changes of the stream function}
\label{bdy:section}

The stream function from Proposition~\ref{stm:stream} is defined even when the eigenspinor vanishes. Now, we describe its behavior near a boundary zero of the eigenspinor. The main result of this section is the following Lemma \ref{bdy:oscillation}, which is a boundary-level partial converse to Lemma \ref{stm:zerofree}: we show that, near a boundary zero of the eigenspinor, the stream function necessarily takes both signs in every neighborhood of zero. This, together with Lemma \ref{stm:zerofree}, helps to prove the persistence of non-vanishing property of the first positive eigenspinor during deformation of domains; see Section \ref{con:section} for details.

Note that, by IMB, a boundary zero of either component of the eigenspinor is necessarily a common zero. In the sequel, the order of such a common zero refers to the least total degree of a nonzero homogeneous term in the Taylor expansions of its components; in particular, as proved in the following, this order is always finite.

\begin{lemma}\label{bdy:oscillation}
Let $\psi=(u,v)^{\top}$ be an eigenspinor associated with the positive eigenvalue $\lambda>m\geq0$. If $z_0\in\partial\Omega$ is a common zero of $\psi$, there exists a smooth conformal upper-half-disk chart $\Psi$ with $\Psi(0)=z_0$, an integer $n\geq1$, and a complex number $\alpha\ne0$ such that the stream function associated with $\psi$ admits the asymptotics
\begin{equation}\label{bdy:streamexpansion}
 p\circ\Psi(z)=
 \frac{|\alpha|^2}{2n+1}\operatorname{Im}(z^{2n+1})
                                    +O(|z|^{2n+2}).
\end{equation}
In particular, $p$ takes both positive and negative values inside $\Omega$. Consequently, an eigenspinor whose stream function satisfies $p\geq0$ in $\Omega$ has no common boundary zero.
\end{lemma}
The $2n+1$ order of the stream function indicated in \eqref{bdy:streamexpansion} can be intuitively illustrated. Suppose that the upper component $u$ is of order $n$. Then the lower component $v$ is of the same order by the IMB; hence, the product $uv$ is of order $2n$. Recalling that the gradient of $p$ is proportional to $uv$ as seen by \eqref{stm:streamgradient}, the $2n+1$ order of $p$ is then clear.

The remainder of this section is dedicated to the proof of Lemma \ref{bdy:oscillation}.

\subsection{Conformal chart and representation of eigenspinor}
We first construct the conformal chart $\Psi$. This is given by the smooth conformal boundary regularity of $\partial \Omega$, which provides a holomorphic map $\Psi$ from a small upper half-disk
\[
 B_\varepsilon^+=\{z\in\C:|z|<\varepsilon,\ \operatorname{Im}z>0\}
\]
onto a neighborhood of $z_0$ in $\Omega$. It satisfies $\Psi(0)=z_0$, and has nonzero derivative within $B_\varepsilon^+$ by shrinking the radius $\varepsilon$ if necessary. The real diameter of $B_\varepsilon^+$ maps to $\partial\Omega$ and the upper half-disk maps to the interior of $\Omega$.

Next, we list some basic properties of the components $u$ and $v$ under the chart $\Psi$. First, choose a holomorphic square root $h$ of $\Psi^{\prime}$ on the chart and define
\begin{equation}\label{bdy:halfs}
\omega:=|\Psi'|=|h|^2,\quad
 \chi_1:=h(u\circ\Psi),\quad
 \chi_2:=\bar h(v\circ\Psi).
\end{equation}
The branch $h$ is chosen smoothly up to the diameter thanks to the non-vanishing derivative $\Psi'$. Recalling that $a=\lambda+m$ and $b=\lambda-m$ as in \eqref{pre:ab}, the chain rule gives
\begin{equation}\label{bdy:flat}
 \db\chi_1=\frac{\ii a\omega}{2}\chi_2,\quad
 \dz\chi_2=\frac{\ii b\omega}{2}\chi_1.
\end{equation}
Moreover, substituting the IMB into \eqref{bdy:halfs} gives
\begin{equation}\label{bdy:flatmit}
 \chi_2
 =\bar h\,\frac{\Psi'}{|\Psi'|}(u\circ\Psi)
 =h(u\circ\Psi)=\chi_1
                    \quad\hbox{on the real diameter of $B_\varepsilon^+$}.
\end{equation}

\subsection{Reflection and a bounded first-order system}
In this part, we construct the reflection continuation of $(\chi_1,\chi_2)^{\top}$ and present the first-order system satisfied by the continued spinor, which will be applied to study its order near a boundary zero. First, on the lower half-disk of $B_\varepsilon$, we define the reflected fields
\[
 \widehat\chi_1(z):=\chi_2(\bar z),\quad
 \widehat\chi_2(z):=\chi_1(\bar z),
\]
and retain the original fields on the upper half-disk. Note that there is no complex conjugation of their values in this reflection, and only the components and the arguments are interchanged. In particular, thanks to \eqref{bdy:flatmit}, the reflected fields are continuous across the diameter. Moreover, $\widehat\chi_k$ are locally Lipschitz since they are smooth on each closed half-disk and agree on the diameter. To recall the jump of derivatives explicitly, define the bounded measurable coefficients
\[
 a_{\mathrm{ref}}(z):=
 \begin{cases}a\omega(z),&\operatorname{Im}z\geq0,\\
               b\omega(\bar z),&\operatorname{Im}z<0,\end{cases}
 \qquad
 b_{\mathrm{ref}}(z):=
 \begin{cases}b\omega(z),&\operatorname{Im}z\geq0,\\
               a\omega(\bar z),&\operatorname{Im}z<0.\end{cases}
\]
Then it is direct to check that
\[
 \db\widehat\chi_1=\ii a_{\mathrm{ref}}\widehat\chi_2/2,\qquad
 \dz\widehat\chi_2=\ii b_{\mathrm{ref}}\widehat\chi_1/2,
\]
which holds almost everywhere in the weak sense. We will work mainly with the spinor form:
\begin{equation}\label{bdy:vekua}
 \db F=\mathcal B\,\bar F,
\end{equation}
where
\begin{equation}\label{bdy:vector}
 F:=(\widehat\chi_1,\overline{\widehat\chi_2})^{\top},\qquad
 \mathcal B:=
 \begin{pmatrix}0&\ii a_{\mathrm{ref}}/2\\
                -\ii b_{\mathrm{ref}}/2&0\end{pmatrix}\in L^\infty.
\end{equation}

\subsection{An explicit matrix factorization}
In this part, we prove an explicit matrix factorization of the spinor $F$ introduced in \eqref{bdy:vector}:
\begin{equation} \label{eq:F_factorization}
F(z)=Q(z)H(z),
\end{equation}
where $Q(z)$ is continuous with a uniformly bounded inverse almost everywhere in $B_\varepsilon$, and $H(z)$ is holomorphic. To do this, we first introduce
\begin{equation}\label{bdy:coefficient}
 C(z):=
 \begin{cases}
 \displaystyle\frac{(\mathcal B(z)\bar F(z))F(z)^*}{|F(z)|^2},
                      &F(z)\ne0,\\[1ex]
 0,&F(z)=0.
 \end{cases}
\end{equation}
It is clear that $\|C(z)\|\leq\|\mathcal B(z)\|$. Moreover $CF=\mathcal B\bar F$ for $z\in B_\varepsilon$. Hence \eqref{bdy:vekua} is recast as
\begin{equation*}
\db F=CF,
\end{equation*}
i.e., the introduction of the matrix $C$ helps to remove the conjugation. With the notion of $C$, we now construct the matrix $Q$ in the factorization \eqref{eq:F_factorization} by a fixed point argument. For a scalar or matrix function $f$ on $B_\varepsilon$, recall that the Cauchy-Green operator is defined as
\begin{equation}\label{bdy:cauchy}
 (\mathcal T_\varepsilon f)(z):=\frac1\pi\int_{B_\varepsilon}\frac{f(\zeta)}{z-\zeta}\,dA(\zeta).
\end{equation}
It is a right inverse of the derivative $\db$: in the sense of distribution, we know  $\db\mathcal T_\varepsilon f=f$ on $B_\varepsilon$; see \cite[Section 2]{BBC}. Since $B_\varepsilon\subset B_{2\varepsilon}(z)$ for $z\in B_\varepsilon$, one can check
\[
 \|\mathcal T_\varepsilon f\|_\infty
 \leq\frac{\|f\|_\infty}{\pi}
                     \int_{B_{2\varepsilon}(0)}|\zeta|^{-1}\,dA(\zeta)
 =4\varepsilon\|f\|_\infty.
\]
Shrinking $\varepsilon$ if necessary so that $\delta=4\varepsilon\|C\|_\infty<1/3$, the map $Q\mapsto I+\mathcal T_\varepsilon(CQ)$ is then a contraction on the Banach space of bounded matrix fields; we define $Q$ as its unique fixed point. In particular, $Q$ satisfies 
\begin{equation}\label{bdy:gauge}
 Q=I+\mathcal T_\varepsilon(CQ),\quad
\|Q-I\|_\infty\leq\frac{\delta}{1-\delta}<\frac12,
\end{equation}
and
\begin{equation*}
\db Q=CQ
\end{equation*}
in the sense of distributions. By \eqref{bdy:gauge}, $Q$ is invertible almost everywhere with a uniformly bounded inverse. Now we show the regularity of $Q$. Note that the derivative $\dz\mathcal T_\varepsilon$ is the Beurling singular integral, which is bounded on $L^r$ for $1<r<\infty$; see \cite[Section 2, equations (2.9)--(2.11)]{BBC}. Since $CQ$ is bounded and supported in a bounded disk, \eqref{bdy:gauge} implies that $Q\in W^{1,r}_{\rm loc}$ for every finite $r>2$. Then the Sobolev embedding gives a locally Hölder continuous representative of $Q$, which also holds for the inverse $Q^{-1}$.

Finally, we prove the factorization \eqref{eq:F_factorization}. Since $F$ is Lipschitz and bounded, the weak product rule is justified:
\[
 \db(Q^{-1}F)
 =-Q^{-1}(\db Q)Q^{-1}F+Q^{-1}\db F
 =-Q^{-1}CF+Q^{-1}CF=0.
\]
Hence, $H:=Q^{-1}F$ is holomorphic. This concludes the proof.

\subsection{Common order of components and asymptotics}
In this part, we justify the intuition that the two components of an eigenspinor are of the same order near a common boundary zero (see the illustration following Lemma \ref{bdy:oscillation}). To be precise, we will prove that
\begin{equation}\label{bdy:leading}
 \chi_1=\alpha z^n+O(r^{n+1}),\quad
 \chi_2=\alpha\bar z^n+O(r^{n+1}),
 \quad r=|z|,\quad \alpha\ne0,
\end{equation}
for some $n\geq 1$ and $\alpha\neq 0$. Indeed, since the holomorphic factor $H$ in the factorization \eqref{eq:F_factorization} does not vanish identically (otherwise $F\equiv 0$, giving a trivial eigenspinor), we can write near a common zero of the eigenspinor that
\[
 H(z)=z^nH_0(z)\quad \text{with } n\geq 1 \text{ and }  H_0(0)\ne0.
\]
Then, the Hölder continuity and invertibility of $Q$ imply that
\[
 F(z)=z^n c+O(|z|^{n+\beta})\quad \text{with }
 c:=Q(0)H_0(0)\ne0\text{ and } 0<\beta<1.
\]
Write $c=(\alpha,\gamma)^{\top}$; then the leading terms of the entries of $F$ are
$\chi_1\sim\alpha z^n$ and $\chi_2\sim\bar\gamma\bar z^n$, respectively. Noting that the boundary condition \eqref{bdy:flatmit} gives $\alpha=\bar\gamma$ and the Taylor expansion indeed upgrades the exponent $n+\beta$ to $n+1$, the proof of \eqref{bdy:leading} then completes.

\subsection{Expansion of the stream function}

Finally, we prove the desired expansion \eqref{bdy:streamexpansion}, by which the sign change of the stream function is clear. Let $\widetilde p:=p\circ\Psi$. By direct calculation,
\begin{equation}\label{bdy:streampullback}
 2\db\widetilde p
 =\ii(\bar u v)\circ\Psi\,\overline{\Psi^{\prime}}
 =\ii\overline{\chi_1}\chi_2
 =\ii|\alpha|^2\bar z^{2n}+O(r^{2n+1}),
\end{equation}
where the identity $\overline{\Psi'}=\bar h^2$ in \eqref{bdy:halfs} is applied. Noting that $2\db\operatorname{Im}(z^\ell)=\ii\ell\bar z^{\ell-1}$, we know the difference
\[
 R(z):=\widetilde p(z)-
              \frac{|\alpha|^2}{2n+1}\operatorname{Im}(z^{2n+1})
\]
is real and satisfies $|\nabla R(z)|=O(r^{2n+1})$. Since $R(z)$ vanishes at the origin, integrating $\nabla R(z)$ along the segment from $0$ to $z$ in the upper half-disk gives $R(z)=O(r^{2n+2})$, which proves \eqref{bdy:streamexpansion}. To discuss the sign change of $p$, we choose rays with arguments $\theta_+:=\pi/[2(2n+1)]$ and $\theta_-:=3\pi/[2(2n+1)]$, which satisfy $\sin((2n+1)\theta_+)=1$ and $\sin((2n+1)\theta_-)=-1$. Hence, along these rays, the respective positive and negative leading terms in \eqref{bdy:streamexpansion} dominate its remainder for a small radius. This  concludes the proof of Lemma~\ref{bdy:oscillation}.
\section{Conformal continuation of the first positive eigenspinor}
\label{con:section}

The preceding results of Sections \ref{stm:section} and \ref{bdy:section}, particularly Lemma \ref{stm:zerofree} and \ref{bdy:oscillation}, transfer the non-vanishing property of the upper component of an eigenspinor into both an ordering criterion and a stable property of eigenfunctions. In this section, we apply these results along with a deformation argument to show that, for every smooth domain, the first positive eigenspinor always has a non-vanishing upper component, and hence the first positive eigenvalue is simple; see Theorem \ref{con:nonvanishing}.

\subsection{Conformal deformation of domain}
\label{con:path-overview}

In this part, we construct a smooth conformal deformation from a disk to the target domain $\Omega$ and pulls the eigenvalue problems on $\Omega$ back to the unit disk. This provides a fixed setting in which eigenvalues and eigenspinors can be compared throughout the continuation. To start, choose a Riemann map $\Phi:\D\to\Omega$, with $\D$ being the unit disk. By the Kellogg-Warschawski theorem, $\Phi$ extends to the whole closure $\overline\D$ with $\Phi'(z)\ne0$ for $z\in\overline\D$; see \cite{Pommerenke}. Then we define
\begin{equation}\label{con:path}
 F_t(z):=
 \begin{cases}
  [\Phi(tz)-\Phi(0)]/t,&0<t\leq1,\\
  \Phi'(0)z,&t=0,
 \end{cases}
 \qquad \Omega_t:=F_t(\D).
\end{equation}
It is direct to verify that 1) $\Omega_0$ is a disk and $\Omega_1$ is a translate of $\Omega$, 2) $F_{t}$ is injective for $t>0$, and 3) $F_{t}(z)$ depends smoothly on $(t,z)$. This gives the desired deformation path.

Now we construct the pullback of eigenspinor. First, let $\omega_t:=|F_t'|$. Then the compactness of $[0,1]\times\overline\D$ gives constants $c_0,C_0>0$ such that
\begin{equation}\label{con:weights}
 c_0\leq\omega_t(z)\leq C_0,
 \quad \sup_{0\leq t\leq1}\|\omega_t\|_{C^j(\overline\D)}<\infty
 \quad\text{for every integer }j\geq0.
\end{equation}
Hence, a holomorphic logarithm $L$ of $\Phi^{\prime}$ exists on $\D$ and extends smoothly to the boundary. Consequently, the functions
\[
 q_t(z):=\exp\bigl(L(tz)/2\bigr)
\]
form a smooth family of holomorphic square roots of $F_t^{\prime}$. Now, for a spinor $\psi=(u,v)^{\top}$ on $\Omega_t$, we first introduce the \emph{half-spin pullback}
\begin{equation}\label{con:halfspin}
 \chi_1:=q_t(u\circ F_t),\quad
 \chi_2:=\overline{q_t}(v\circ F_t).
\end{equation}
This pullback is not unitary with respect to the Lebesgue measure: it is direct to check that
\begin{equation}\label{con:normalization}
 \|\psi\|_{L^2(\Omega_t)}^2
       =\int_\D\omega_t\bigl(|\chi_1|^2+|\chi_2|^2\bigr).
\end{equation}
Later, a unitary pullback will be constructed by incorporating the factor $\omega_t$. Before that, we list several properties of the spinor $\chi=(\chi_1,\chi_2)^{\top}$. Recalling \eqref{pre:system}, one can check
\begin{equation}\label{con:weighted-equations}
 \db\chi_1=\frac{\ii(\lambda+m)}2\omega_t\chi_2,
 \qquad
 \dz\chi_2=\frac{\ii(\lambda-m)}2\omega_t\chi_1.
\end{equation}
The IMB of $(u,v)^{\top}$ also leads to the boundary condition of $\chi$
\begin{equation}\label{con:fixed-boundary}
 \chi_2=\ii z\chi_1\quad\text{on }\partial\D,
\end{equation}
where one should note that there is no unit tangential appearing. Finally, $\chi$ satisfies the transformed eigenproblem
\begin{equation}\label{con:generalized-eigenproblem}
 (D_0+m\omega_t\sigma_3)\chi=\lambda\omega_t\chi.
\end{equation}

Now, we introduce the \emph{unitary pullback}
\begin{equation*}
\widehat\chi=\omega_t^{1/2}\chi .
\end{equation*}
It is clear that $\widehat \chi$ satisfies the same boundary condition as \eqref{con:fixed-boundary}. Thus, on the fixed Hilbert space $\mathcal H=L^2(\D;\C^2)$, the original Dirac operator is unitarily equivalent to
\begin{equation}\label{con:fixed-operator}
 A_t:=\omega_t^{-1/2}D_0\omega_t^{-1/2}+m\sigma_3,
 \qquad
 \operatorname{Dom}A_t=E,
\end{equation}
where
\begin{equation}\label{con:fixed-domain}
 E:=\{\widehat\chi\in H^1(\D;\C^2):
          \Tr\widehat\chi_2=\ii z\Tr\widehat\chi_1\text{ on }\partial\D\}.
\end{equation}

\subsection{Spectral continuity}

In this part, we justify that the operator family \eqref{con:fixed-operator} is continuous in the norm-resolvent sense, which then leads to the continuity of the first positive eigenvalue.

\begin{lemma}\label{con:spectral}
The self-adjoint operators $A_t$ have compact resolvent and depend continuously on $t$ in the norm-resolvent sense. Consequently, the map $t\longmapsto\lambda_1^+(\Omega_t;m)$ is continuous on $[0,1]$.
\end{lemma}

Note that the continuity of $\lambda_1^+$ follows by a standard perturbation argument based on the norm-resolvent continuity (see, e.g., \cite[Chapter VII]{Kato}). The key points are 1) $\lambda_1^{+}$ does not reach zero along the deformation by Lemma \ref{pre:gap}, and 2) the possible collision between $\lambda_1^{+}$ and $\lambda_2^{+}$ hurts only the differentiability, not the continuity. The details are omitted here. In the following, we only prove the norm-resolvent continuity.

\begin{proof}[Proof of Lemma \ref{con:spectral}]
The self-adjointness of $A_{t}$ and the compactness of the resolvent follow from that of the original Dirac operator (see Section \ref{sec:general_spectrum}) and the unitary equivalence. Moreover, the continuity of $t\mapsto A_t$ in $\mathcal L(E,\mathcal H)$ follows directly by that of the coefficients in \eqref{con:fixed-operator}. Now we justify the norm-resolvent continuity. Fix $t_0 \in [0,1]$ and define
\[
 B_{t,t_0}:=(A_t-A_{t_0})(A_{t_0}-\ii)^{-1}.
\]
Then $\|B_{t,t_0}\|_{\mathcal L(\mathcal H)}\to0$ as $t\to t_0$; hence, for sufficiently small $t-t_0$
\[
 (A_t-\ii)^{-1}=(A_{t_0}-\ii)^{-1}(I+B_{t,t_0})^{-1},
\]
which justifies the norm-resolvent continuity together with the $H^1$ boundedness of resolvent $(A_{t_0}-\ii)^{-1}\in L(\mathcal H,E)$ (see \cite{BFSVSA}).
\end{proof}
We also record the uniform regularity and smooth compactness of normalized eigenspinors along the deformation. This enables the closedness argument used later (in Section \ref{sec:simple_nonvanishing}) to prove the non-vanishing property of the first positive eigenspinor.

\begin{lemma}\label{con:compactness}
Suppose $t_j\to t_*\in [0,1]$ and $\psi_j$ are $L^2$-normalized IMB eigenspinors on $\Omega_{t_j}$, with eigenvalues $\lambda_j$ in a bounded set. Then the associated half-spin pullbacks $\chi_{j}$ are uniformly bounded in $H^s(\D)$ for every $s\geq 0$. Passing to a subsequence for which $\lambda_j\to\lambda_*$ if necessary, $\chi_{j}$ converge in $C^r(\overline\D)$ to a smooth, nonzero normalized eigenspinor pullback at $t_*$ and $\lambda_*$ for all $r\geq 0$. The same compactness holds for the unitary pullbacks $\hat{\chi}_{j}$.
\end{lemma}
\begin{proof}
Here, we only prove the claims for $\chi_{j}$, while $\hat{\chi}_j$ are treated  similarly. First, suppressing the subscript $j$, the regularity of $\chi$ follows by bootstrapping based on the equation \eqref{con:weighted-equations} and \eqref{con:fixed-boundary}; see \cite[Theorem 3.9]{BB2016}. Then the remaining compactness property follows by Sobolev embedding. Finally, noting that all coefficients in \eqref{con:weighted-equations} and \eqref{con:fixed-boundary} converge smoothly as $t\to t_*$, so the limiting pair $(\lambda_*,\chi_*)$ indeed solves \eqref{con:weighted-equations} and \eqref{con:fixed-boundary} at $t= t_*$.
\end{proof}

\subsection{First positive eigenvalue of the disk and the associated eigenspinor} \label{sec:disk_first_eigenspinor}

In this part, we construct the explicit disk eigenspinor with a non-vanishing upper component; consequently, Lemma \ref{stm:zerofree} identifies its eigenvalue as the first positive one and implies its simplicity. This supplies the starting point for the open-and-closed continuation argument used to prove Theorem \ref{con:nonvanishing} later.

\begin{lemma}\label{con:disk}
Let $\D_R$ be a disk of radius $R>0$ and let $m\geq0$. There is a
unique number $x\in(0,j_{0,1})$ satisfying
\begin{equation}\label{con:disk-root}
 \frac{J_1(x)}{J_0(x)}
       =\frac{\sqrt{x^2+(mR)^2}+mR}{x}.
\end{equation}
Then $\lambda=\sqrt{m^2+x^2/R^2}$ is the first positive eigenvalues of the IMB Dirac operator on $\D_{R}$, which is simple, and the associated eigenspinor has a strictly positive upper component on the closed disk.
\end{lemma}
\begin{proof}
By Lemma~\ref{pre:bessel}, $J_1/J_0$ is strictly increasing from zero to infinity on $(0,j_{0,1})$. On the other hand, the right side of \eqref{con:disk-root} is positive and decreasing (strictly if $m>0$, while it is identically one if $m=0$). Its limit at the left endpoint is infinity when $m>0$ and one when $m=0$, while the right-endpoint limit is always finite. Then, the existence of a unique root $x$ follows by the intermediate value theorem.

Now, we translate the center of the disk to the origin, let $k=x/R$, and define in polar coordinates 
\begin{equation}\label{con:disk-spinor}
 u(r,\theta):=J_0(kr),\quad
 v(r,\theta):=\ii e^{\ii\theta}
               \frac{k}{\lambda+m}J_1(kr).
\end{equation}
Then, it is easy to check that $\psi=(u,v)^{\top}$ is an eigenspinor associated with $\lambda$. Since $kR<j_{0,1}$, the upper component $u$ is strictly positive on the closed disk. Then applying Lemma~\ref{stm:zerofree}
concludes the proof.
\end{proof}

\subsection{Simplicity and global non-vanishing upper component} \label{sec:simple_nonvanishing}

We now prove the main result of this section, using the ingredients from the preceding subsections.

\begin{theorem}\label{con:nonvanishing}
Let $\Omega$ be a bounded smooth, simply connected planar domain and $m\geq0$. The first positive eigenvalue of the IMB Dirac operator is simple, and the associated eigenspinor has a non-vanishing upper
component over $\overline\Omega$.
\end{theorem}
\begin{proof}
Let $F_{t}(z)$ be the deformation of the domains defined by \eqref{con:path} with the right endpoint being $\Omega$. Now, we define $\mathcal S$ to be the set of $t\in[0,1]$ for which the first-positive eigenspace on $\Omega_t$ contains an eigenspinor whose upper component is nonzero on $\overline{\Omega_t}$. Then we know by Lemma~\ref{con:disk} that $0\in\mathcal S$. In the sequel, we will prove that $\mathcal{S}$ is both open and closed in $[0,1]$, by which we conclude that $\mathcal{S}=[0,1]$ and completes the proof.

{\color{blue}Step 1.} We first prove the relative openness. Let $t_0\in\mathcal S$, for which Lemma~\ref{stm:zerofree} proves that the first positive eigenvalue there is simple. Then the standard perturbation theory tells that the simple eigenvalue continues to a neighborhood of $t_0$ with a $L^2$ continuous family of normalized eigenspinors $t\mapsto\psi_{t}$. Note that the continued eigenvalue must be the first positive one by shrinking the neighborhood if necessary and recalling the gap condition in Lemma~\ref{con:spectral}. The claim on the non-vanishing upper component of $\psi_{t}$ is justified by noting that Lemma~\ref{con:compactness} upgrades the $L^2$ continuity of $t\mapsto\psi_{t}$ to $C^0$.

{\color{blue}Step 2.1} To show the relative closedness, suppose that $t_j\in\mathcal S$ converges to $t_*$. Choose normalized first-positive eigenspinors $\psi_j=(u_j,v_j)^{\top}$ with non-vanishing $u_j$. By Lemmas~\ref{con:spectral} and \ref{con:compactness}, a subsequence of their half-spin pullbacks smoothly converges to a normalized first-positive eigenspinor pullback on $\Omega_{t_*}$. Denoting the corresponding (unpulled) spinor by $\psi_{*}=(u_*,v_*)^{\top}$, we now prove the density $U_{*}=|u_{*}|^2$ is non-vanishing, which then completes the proof. This is achieved again by exploiting the stream function.

{\color{blue}Step 2.2} Letting $p_*$ and $p_j$ be the stream functions associated with $\psi_{*}$ and $\psi_j$, respectively, we first prove
\begin{equation} \label{eq:positive_stream}
p_*\geq0 \quad \text{on $\Omega_{t_*}$}.
\end{equation}
In fact, if we define $P_j:=p_j\circ F_{t_j}$ (i.e., pullback of the streams), one can check
\begin{equation}\label{con:stream-pullback}
 2\db P_j=\ii\overline{\chi_{1,j}}\chi_{2,j},
 \quad P_j=0\text{ on }\partial\D.
\end{equation}
Since the right-hand side of this first order equation, as well as the boundary condition, converges smoothly, we conclude $P_j\to P_*=p_*\circ F_{t_*}$. Since $p_j>0$ by Lemma~\ref{stm:zerofree}, the proof of \eqref{eq:positive_stream} then completes.

{\color{blue}Step 2.3} By \eqref{eq:positive_stream} and Lemma~\ref{bdy:oscillation}, we know that $\psi_{*}$ does not vanish on the boundary, which leads to
\begin{equation}\label{con:limit-boundary-minimum}
 c_*:=\min_{\partial\Omega_{t_*}}|u_*|^2>0.
\end{equation}
We now exclude the possible interior zero of $u_{*}$, which then concludes the entire proof. Write $U_j=|u_j|^2$ and $a_j=\lambda_1^+(\Omega_{t_j};m)+m$. Note that, by Lemma~\ref{stm:zerofree},
\[
 U_j-a_jp_j\geq\min_{\partial\Omega_{t_j}}U_j.
\]
Passing both sides to the limit leads to
\[
 U_*-a_*p_*\geq c_*>0
       \quad\text{on }\overline{\Omega_{t_*}}
\]
with $a_* =\lambda_1^+(\Omega_{t_*};m)+m>0$ and $p_{*}\geq 0$ by \eqref{eq:positive_stream}. This implies that $U_{*}$ is non-vanishing and concludes the whole proof.
\end{proof}
\section{The eigenspinor map and the comparison of eigenvalues: proof of Theorem \ref{main:fk}}\label{sec:application}

In this section, we prove the strict Faber-Krahn inequality. The proof is based on Theorem \ref{con:nonvanishing}: thanks to the nonvanishing upper component $u$ of the first positive eigenspinor, we are able to construct the sphere map 
\begin{equation*}
N:=\frac{(2\Rea\eta,\,2\Ima\eta,\,1-|\eta|^2)}{1+|\eta|^2}\quad \text{ with }\eta:=\frac{\sqrt{a}v}{\sqrt{b}\overline{u}};
\end{equation*}
see \eqref{app:N} for details. In particular, using the Dirac equation, we show in Section \ref{sec:main_proof_2} that the map $N$ satisfies
\begin{equation*}
|\nabla N|^2+2J_N=4k^2
\end{equation*}
with $J_N$ being the spherical Jacobian associated with $N$ (see \eqref{app:jacobian}), and $N$ has degree one over the northern cap; see Section \ref{sec:main_proof_3}. These properties enable us to apply the following geometric result:
\begin{theorem}
\label{cap:sharp}
Let $\Omega\subset\R^2$ be a bounded connected smooth domain of area $A$, and $N:\overline\Omega\to S^2$ be a smooth map that is real analytic in the interior. Suppose, for $k>0$ and $-1<c<1$, that
\begin{equation}
 |\nabla N|^2+2J_N=4k^2\quad\hbox{in }\Omega,
 \quad N_3=c\quad\hbox{on }\partial\Omega,
 \label{cap:forward}
\end{equation}
where $J_N$ is the spherical Jacobian associated with $N$. Assume that $\deg(N,\Omega,y)=1$ for every regular target value $y$ with $y_3>c$. Then
\begin{equation}
 k^2A\geq\pi x_c^2,
 \qquad \text{ with $x_c\in (0,j_{0,1})$ solves uniquely }\frac{J_1(x_c)}{J_0(x_c)}=\sqrt{\frac{1-c}{1+c}}.
 \label{cap:comparison}
\end{equation}
When the equality holds, it follows that $\Omega$ is a disk.
\end{theorem}
Both an intuitive explanation of Theorem \ref{cap:sharp} and the detailed proof are presented in Section \ref{sec:cap}. In particular, the inequality \eqref{cap:comparison} implies \(\lambda\ge\lambda_1^+(B_{|\Omega|};m)\), which leads to the conclusion of Theorem \ref{main:fk}. The detailed argument is presented in the sequel.

\subsection{Sphere map constructed by the eigenspinor} \label{sec:main_proof_1}
In this part, we construct the sphere map $N$ and list its basic properties. Fix the mass $m\geq0$ and let $\lambda=\lambda_1^+(\Omega;m)$. Thanks to Theorem~\ref{con:nonvanishing}, we can select a normalized eigenspinor $(u,v)^{\top}$ whose upper component is nonzero on the $\overline{\Omega}$. Recall the notation $a=\lambda+m$, $b=\lambda-m$, $k=\sqrt{ab}$, and $U=|u|^2$, $V=|v|^2$. Then we define the rescaled functions
\begin{equation}\label{app:rescale}
 f:=\sqrt b\,u,\quad g:=\sqrt a\,v,
 \quad \varrho:=|f|^2+|g|^2=bU+aV>0.
\end{equation}
Correspondingly, \eqref{pre:system} becomes
\begin{equation}\label{app:massless}
 \db f=\frac{\ii k}{2}g,\qquad
 \dz g=\frac{\ii k}{2}f.
\end{equation}
With these notions, we define
\begin{equation} \label{app:N}
N:=\frac{(2\Rea\eta,\,2\Ima\eta,\,1-|\eta|^2)}{1+|\eta|^2}
\end{equation}
with
\begin{equation} \label{app:eta}
\eta:=\frac{g}{\bar f},
\end{equation}
which is well defined since $f$ is non-vanishing. The map $N$ takes values on the unit sphere $S^2\subset\R^3$ and omits its south pole. Since $f$ and $g$ solve constant-coefficient Helmholtz equations as seen by \eqref{app:massless}, $N$ is real analytic in the interior of $\Omega$ and smooth throughout the closure.

\subsection{The energy and Jacobian identity} \label{sec:main_proof_2}
This subsection is dedicated to proving the first identity of \eqref{cap:forward}. To this end, we introduce the stereographic parametrization
\begin{equation}\label{app:stereo}
 \mathcal S(z):=
 \frac{(2\Rea z,2\Ima z,1-|z|^2)}{1+|z|^2}
 \quad(z\in\C).
\end{equation}
It is clear that $\mathcal S$ is orientation-preserving and maps  $z=0$ to the north pole. Writing $z=X+\ii Y$, direct differentiation gives the spherical metric and area form in this chart:
\begin{equation}\label{app:metric}
 ds_{S^2}^2=\frac{4(dX^2+dY^2)}{(1+X^2+Y^2)^2},\quad
 \omega_{S^2}=\frac{4\,dX\wedge dY}{(1+X^2+Y^2)^2}.
\end{equation}
Here, $dX\wedge dY$ is the oriented area two-form. Using this parametrization, it is now direct to calculate the square norm of gradient $|\nabla N|^2=|N_x|^2+|N_y|^2$ and the spherical Jacobian $J_N:=N\cdot(N_x\times N_y)$: by \eqref{app:metric} and $|\eta_x|^2+|\eta_y|^2=2(|\eta_z|^2+|\eta_{\bar z}|^2)$, it holds that
\begin{align}
 |\nabla N|^2&=
 \frac{8(|\eta_z|^2+|\eta_{\bar z}|^2)}{(1+|\eta|^2)^2},
                                                               \label{app:energy}\\
 J_N&=\frac{4(|\eta_z|^2-|\eta_{\bar z}|^2)}{(1+|\eta|^2)^2}.
                                                               \label{app:jacobian}
\end{align}
Moreover, differentiating \eqref{app:eta} and applying \eqref{app:massless} give
\begin{align}
 \eta_z
 &=\frac{g_z\bar f-g(\bar f)_z}{\bar f^2}
 =\frac{(\ii kf/2)\bar f-g(-\ii k\bar g/2)}{\bar f^2}
 =\frac{\ii k\varrho}{2\bar f^2}.                 \label{app:quotientderivative}
\end{align}
Since $1+|\eta|^2=\varrho/|f|^2$, this implies that
\begin{equation}\label{app:forward}
 \frac{|\eta_z|^2}{(1+|\eta|^2)^2}=\frac{k^2}{4},
 \quad |\nabla N|^2+2J_N=4k^2,
\end{equation}
which completes the proof.

\subsection{Boundary latitude and degree one} \label{sec:main_proof_3}
In this part, we prove the second identity of \eqref{cap:forward} and justify $N$ has degree one, which are assumed in Theorem \ref{cap:sharp}. First, the IMB gives $|v|=|u|$ on $\partial\Omega$, and therefore
\begin{equation}\label{app:latitude}
 N_3=\frac{b-a}{b+a}=-\frac m\lambda=:c\in(-1,0],
 \quad
 \eta=\sqrt{\frac ab}\,\ii\nu\frac{u}{\bar u}.
\end{equation}
Next, we calculate explicitly the degree of $N$. Let $\zeta$ be a regular value of $\eta$ with $|\zeta|<r_c:=\sqrt{a/b}=\sqrt{(1-c)/(1+c)}$. First, note that the degree of $N$ is equal to the winding number $\operatorname{wind}(\eta|_{\partial\Omega}-\zeta,0)$, which is checked using the Stokes theorem. On the other hand, since the homotopy $\eta|_{\partial\Omega}-s\zeta$ with $0\leq s\leq1$ is non-vanishing, due to the choice $|\zeta|<r_c$, it suffices to prove that
\begin{equation}\label{app:boundarywind}
 \operatorname{wind}(\eta|_{\partial\Omega},0)=1
\end{equation}
to conclude that $N$ has degree one. Indeed, since $u$ is non-vanishing, it has a global smooth real phase $\theta$, with $u=|u|e^{\ii\theta}$, by which \eqref{app:boundarywind} follows by direct calculation.

\subsection{Proof of \eqref{main:ineq}}
Now we have verified all hypotheses of Theorem~\ref{cap:sharp}. It gives, with $R=\sqrt{A/\pi}$,
\begin{equation}\label{app:comparison}
 kR\geq x_{-m/\lambda}, \quad \text{ with $x_{-m/\lambda}\in (0,j_{0,1})$ solves uniquely }
 \frac{J_1(x_{-m/\lambda})}{J_0(x_{-m/\lambda})}
       =\sqrt{\frac{\lambda+m}{\lambda-m}}.
\end{equation}
For fixed $m\geq0$ and $R>0$, define on $\lambda>m$
\[
 L(\lambda):=R\sqrt{\lambda^2-m^2},\qquad
 X(\lambda):=x_{-m/\lambda}.
\]
We claim that $L-X$ is strictly increasing with its unique zero being $\lambda_1^+(B_A;m)$, i.e., the first positive eigenvalue on the disk; then \eqref{app:comparison} implies $\lambda\geq\lambda_1^+(B_A;m)$. Indeed, since $L$ is clearly increasing, it suffices to show that $X$ is non-increasing to conclude the monotonicity: this is actually justified using \eqref{app:comparison} and similar arguments as in Lemma \ref{con:disk}. The existence of a unique zero follows by noting that $L-X$ is negative near $m$ and tends to infinity as $\lambda\to\infty$, which is also checked elementarily. Finally, the fact that the unique zero of $L-X$ is indeed equal to $\lambda_1^+(B_A;m)$ is justified by using the relations
\[
 J_0(kR)=\frac{k}{\lambda+m}J_1(kR),\quad k=\sqrt{\lambda^2-m^2},
\]
and $(\lambda+m)/k=\sqrt{(\lambda+m)/(\lambda-m)}$.

Suppose that $\lambda=\lambda_1^+(B_A;m)$. Then $L(\lambda)=X(\lambda)$ and $k^2A=\pi x_c^2$. The equality statement in Theorem~\ref{cap:sharp} therefore forces $\Omega$ to be a disk. This completes the proof of Theorem~\ref{main:fk}.

\section{Proof of Theorem \ref{cap:sharp}}
\label{sec:cap}

Before showing the proof, we give an illustration of Theorem \ref{cap:sharp}. First, since we have fixed degree one of the map $N$, it means $N$ has a net turn around each latitude. On the other hand, the differential identity
\begin{equation*}
|\nabla N|^2+2J_N=4k^2
\end{equation*}
guarantees that any excess perimeter of the preimage of $N$ gains more area between neighboring levels. As a consequence, the circular level sets of disk cover the northern cap with the least planar area. In fact, the detailed proof adapts the classical level-set comparison method of Talenti (cf. \cite[Section 3-4]{Talenti}) and Chiti (cf. \cite{Chiti}) to sphere-valued maps, using the signed degree, provided by our Dirac model, to close the flux estimate and obtain the sharp radial Bessel comparison.

\subsection{Preliminaries}
We first present some preliminaries for the proof, including the definition of the spherical Jacobian and its relation to the degree. The results in this subsection are standard.

In the sequel, we orient the plane by $dx\wedge dy$, and the unit sphere $S^2=\{y\in\R^3:|y|=1\}$ by its outward normal. For a smooth map $N:\overline\Omega\to S^2$, its signed Jacobian is
\begin{equation}
 J_N=N\cdot(N_x\times N_y),\qquad
 N^*\omega_{S^2}=J_N\,dx\wedge dy.
 \label{cap:jacobian}
\end{equation}
At a regular value $y\notin N(\partial\Omega)$, i.e., $dN$ has rank two at every $N$-preimage of $y$, we know that the Jacobian is non-vanishing. Then the signed degree is defined by
\begin{equation}
 \deg(N,\Omega,y):=
       \sum_{x\in\Omega:\,N(x)=y}\operatorname{sign}J_N(x),
 \label{cap:degree}
\end{equation}
which extends to a locally constant function of $y$ on $S^2\setminus N(\partial\Omega)$. It can also be written in integral form: if $U\subset S^2\setminus N(\partial\Omega)$ is open with a boundary of spherical area zero, then
\begin{equation}
 \int_{N^{-1}(U)}J_N\,dx
       =\int_U\deg(N,\Omega,y)\,dA_{S^2}(y).
 \label{cap:areaformula}
\end{equation}
See \cite{Federer} for a detailed discussion.

In the sequel, away from the poles, we parameterize a sphere point as
\begin{equation} \label{eq:varphi_def}
y=(\sqrt{1-s^2}\cos\varphi,
       \sqrt{1-s^2}\sin\varphi,s),\quad -1<s<1.
\end{equation}
The azimuth $\varphi$ is defined modulo $2\pi$. Direct differentiation
gives the form of the metric and oriented area
\begin{equation}
 d\ell_{S^2}^{\,2}=\frac{ds^2}{1-s^2}+(1-s^2)d\varphi^2,
 \qquad \omega_{S^2}=d\varphi\wedge ds.
 \label{cap:metric}
\end{equation}
We denote $C_t:=\{y\in S^2:y_3>t\}$ by the northern cap, which has area $2\pi(1-t)$ for $-1<t<1$.

\subsection{Level sets, coarea, and the signed winding} \label{sec:level sets}
In this part, as in the first step of the proof, we introduce the latitude superlevel sets, the associated area variation, and convert the degree one condition into an exact total winding of \(2\pi\) along each regular level.

Let $w:=N_3$. The degree one assumption ensures that every regular value $y$ (of $N$) in the cap $C_c$ indeed has a preimage. Taking a sequence of such regular values that tend to the north pole, the compactness implies that $\max w=1$, where the maximum is attained in the interior because $w=c<1$ on the boundary. Thus we know $w$ is non-constant. Then the level sets are defined as follows: for $c<t<1$, let
\begin{equation}
 E_t:=\{x\in\Omega:w(x)>t\},\quad
 \mathcal V(t):=|E_t|,
 \label{cap:volume}
\end{equation}
and at a regular value $t$ of $w$, we define
\begin{equation*}
\Gamma_t:=\partial E_t.
\end{equation*}
Along a level curve, the following quantities are defined to measure the steepness of $w$ and the rate at which the enclosed area changes
\begin{equation}  \label{cap:levelquantities}
X(t):=\int_{\Gamma_t}g\,ds,\quad
 Y(t):=\int_{\Gamma_t}\frac{1}{g}\,ds,\quad
 \text{ with }  g:=|\nabla w|\big|_{\Gamma_t} .
\end{equation}
Here and below $ds$ denotes the planar arclength. The level curve $\Gamma_{t}$ is compactly contained in $\Omega$ and is a finite union of smooth closed curves, on which $g>0$.

Now, we consider the rate of variation of the superlevel-set area, i.e., \(\mathcal V^{\prime}(t)\). Note that the critical set $\{\nabla w=0\}$ has zero measure since $w$ is real analytic and is non-constant as justified before. Hence, the coarea formula shows that, for $c<\alpha<\beta<1$,
\[
 \mathcal V(\alpha)-\mathcal V(\beta)
   =\int_{\{\alpha<w<\beta\}}1\,dx
   =\int_\alpha^\beta
       \left(\int_{\{w=t\}}\frac{1}{|\nabla w|}\,ds\right)dt.
\]
Consequently, we know $\mathcal V$ is locally absolutely continuous on $(c,1)$ and
\begin{equation}
 -\mathcal V^{\prime}(t)=Y(t)\quad\hbox{for almost every }t\in(c,1).
 \label{cap:coarea}
\end{equation}

Finally, we show that, for a regular level, the degree one condition converts to the following winding identities:
\begin{equation}
 \int_{\Gamma_t}\varphi_s\,ds=2\pi,
 \quad \int_{E_t}J_N\,dx=2\pi(1-t),
 \label{cap:winding}
\end{equation}
where $\varphi$ is introduced in \eqref{eq:varphi_def}, the unit positive tangent $\tau_t:=\J\nu_t$ with $\nu_t=-\frac{\nabla w}{g}$ being the outward normal, and the subscript $s$ denotes differentiation along $\tau_t$. To show this, note that the sphere coordinates in
\eqref{cap:metric} now give
\begin{equation}
 J_N=g\varphi_s\quad\hbox{on }\Gamma_t.
 \label{cap:leveljacobian}
\end{equation}
On the other hand, the following one-form is smooth on the sphere minus the south pole
\[
 \alpha:=(1-s)d\varphi
        =\frac{y_1\,dy_2-y_2\,dy_1}{1+y_3}
\]
and $d\alpha=\omega_{S^2}$ (see \eqref{cap:metric}). Thus, recalling that $N(E_t)\subset C_t$ and applying Stokes' theorem, \eqref{cap:areaformula} yields
\[
 (1-t)\int_{\Gamma_t}\varphi_s\,ds
 =\int_{\Gamma_t}N^*\alpha
 =\int_{E_t}J_N\,dx=2\pi(1-t),
\]
where \eqref{cap:areaformula} and the degree one condition are applied to obtain the last equality. This concludes the proof of \eqref{cap:winding}.

\subsection{A perimeter inequality} \label{sec:perimeter inequality}

Based on the results obtained in Section \ref{sec:level sets}, the target of this part is to prove the following inequality, which will justify the intuition stated at the beginning of this section on the excess-area gaining
\begin{equation}
 -\mathcal V^{\prime}(t)\geq
 \frac{2\mathcal V(t)}
 {2k\sqrt{(1-t^2)\mathcal V(t)/\pi}-(1-t^2)}
 \quad\hbox{for almost every }c<t<1.
 \label{cap:volumeode}
\end{equation}
In fact, using \eqref{cap:metric} and \eqref{cap:leveljacobian}, the exact identity $|\nabla N|^2+2J_N=4k^2$ in \eqref{cap:forward} transfers to the level set $\Gamma_t$ and becomes
\begin{equation}
 \left(\frac{g}{\sqrt{1-t^2}}+\sqrt{1-t^2}\,\varphi_s\right)^2
                 +(1-t^2)\varphi_{\nu_t}^{\,2}=4k^2.
 \label{cap:square}
\end{equation}
Hence, the first bracket is bounded by $2k$. Multiplicated by $\sqrt{1-t^2}$ and integrating over $\Gamma_{t}$, the definition \eqref{cap:levelquantities} and the exact winding identity \eqref{cap:winding} then lead to
\begin{equation}
 X(t)+2\pi(1-t^2)\leq2k\sqrt{1-t^2}\,\mathcal P(t).
 \label{cap:flux}
\end{equation}
with $\mathcal P(t)$ being the arclength of $\Gamma_{t}$. By Cauchy-Schwarz, it is clear that $\mathcal P(t)^2\leq X(t)Y(t)$, and hence
\begin{equation}
 Y(t)\geq\frac{\mathcal P(t)^2}{X(t)}
 \geq\frac{\mathcal P(t)^2}
                 {2k\sqrt{1-t^2}\,\mathcal P(t)-2\pi(1-t^2)}.
 \label{cap:perimeterratio}
\end{equation}
By \eqref{cap:coarea}, it now suffices to replace $\mathcal P(t)$ with $2\sqrt{\pi\mathcal{V}}$ on the right side of \eqref{cap:perimeterratio} to conclude the proof of \eqref{cap:volumeode}. This is indeed justified by noting that
\begin{itemize}
    \item the following function is strictly increasing for $P>2\pi\sqrt{1-t^2}/k$
\begin{equation} \label{eq:Ft_def}
F_t(P):=\frac{P^2}{2k\sqrt{1-t^2}\,P-2\pi(1-t^2)}.
\end{equation}
    \item By the elementary inequality $|\nabla N|^2\geq2|J_N|$, together with \eqref{cap:forward}, we have $J_N\leq k^2$; hence \eqref{cap:winding} tells
    \begin{equation}
 k^2\mathcal V(t)\geq2\pi(1-t).
 \label{cap:signedareabound}
\end{equation}
    \item Finally, the planar isoperimetry applied to the complete smooth set $E_t$ gives
    \[
 \mathcal P(t)\geq P_0(t):=2\sqrt{\pi\mathcal V(t)}
 \geq\frac{2\pi}{k}\sqrt{2(1-t)}
 >\frac{2\pi}{k}\sqrt{1-t^2}.
\]
\end{itemize}
Hence, the proof of \eqref{cap:volumeode} is complete.

\subsection{Proof of \eqref{cap:comparison}}
In this step, we convert the differential inequality \eqref{cap:volumeode} to the Bessel-form as stated in Theorem \ref{cap:sharp}. This will prove the inequality part of Theorem \ref{cap:sharp}. In the sequel, we write
\begin{equation} \label{eq:r_def}
\theta:=\arccos t,\quad \theta_0:=\arccos c\in(0,\pi),\quad
 r(\theta):=k\sqrt{\mathcal V(\cos\theta)/\pi}.
\end{equation}
Then, by \eqref{cap:volumeode}, \eqref{cap:signedareabound} and the chain rule, we know
\begin{equation}
 r^{\prime}(\theta)\geq \mathcal F_\theta(r(\theta)),\quad
 \mathcal F_\theta(x)=\frac{x}{2x-\sin\theta},\quad
 r(\theta)\geq2\sin(\theta/2).
 \label{cap:radiusode}
\end{equation}
In particular, the denominators at $x=r(\theta)$ are strictly positive.

On the other hand, recall from Lemma~\ref{pre:bessel} that $h(x)=J_1(x)/J_0(x)$ is strictly increasing from zero to infinity on $(0,j_{0,1})$ and satisfies $h^{\prime}=1+h^2-h/x$. Hence, if we write
\[
 \theta_B(x):=2\arctan h(x),\quad 0<x<j_{0,1},
\]
we know $\theta_B(x)$ increases from zero to $\pi$ with the derivative
\[
 \theta_B^{\prime}(x)=2-\frac{\sin\theta_B(x)}{x}>0.
\]
Denoting its inverse by $r_B:(0,\pi)\to(0,j_{0,1})$, it is straightforward to calculate that
\begin{equation}
 r_B^{\prime}(\theta)=\frac{r_B(\theta)}{2r_B(\theta)-\sin\theta}
             =\mathcal F_\theta(r_B(\theta)),
 \quad \lim_{\theta\to 0^{+}}r_B(\theta)=0,
 \label{cap:diskode}
\end{equation}
with the denominator being positive by $\theta_B^{\prime}(x)>0$.

We are now ready to prove \eqref{cap:comparison}. On any subset $[\varepsilon,T]\subset(0,\theta_0)$, the function $Z(\theta):=\max\{r_B-r,0\}$ is absolutely continuous. Moreover, when $Z>0$, \eqref{cap:radiusode} and \eqref{cap:diskode} give $Z^{\prime}\leq\mathcal F_\theta(r_B)-\mathcal F_\theta(r)\leq0$ (as it is easy to check that $F_\theta$ is decreasing). When $Z=0$, one can verify by the chain rule that $Z^{\prime}=0$ holds almost everywhere. In conclusion,
\[
 Z(T)\leq Z(\varepsilon)\leq r_B(\varepsilon).
\]
Letting $\varepsilon\to 0^{+}$ leads to
\begin{equation}
 r(\theta)\geq r_B(\theta),\quad 0<\theta<\theta_0.
 \label{cap:radiuscomparison}
\end{equation}
Using \eqref{cap:volume}, \eqref{eq:r_def} and noting that $E_t$ increases to $\{w>c\}$ as $t\to c^{+}$, we conclude from \eqref{cap:radiuscomparison} that
\[
 k^2A\geq k^2|\{w>c\}|
       \geq\pi r_B(\theta_0)^2.
\]
This proves \eqref{cap:comparison} by recalling that $x_c=r_B(\theta_0)$ obeys $h(x_c)=\tan(\theta_0/2)=\sqrt{(1-c)/(1+c)}$.

\subsection{Equality and sharpness}

Finally, we prove the equality part of Theorem \ref{cap:sharp}. Suppose that the equality in \eqref{cap:comparison} holds. Then we know $|\{w>c\}|=A$ and the non-negative difference $d:=r-r_B$ vanishes at $\theta_0$. On every interval bounded away from zero, the subtraction of \eqref{cap:diskode} from \eqref{cap:radiusode} gives
\begin{equation}
 d^{\prime}\geq-q(\theta)d,\quad \text{ with }
 q(\theta):=\frac{\sin\theta}
 {(2r-\sin\theta)(2r_B-\sin\theta)}.
 \label{cap:equalityode}
\end{equation}
One can check that, for every $\varepsilon>0$, $q$ is bounded on
$[\varepsilon,\theta_0)$. Hence, the first-order equation \eqref{cap:equalityode} indicates that a positive value of $d$ cannot decrease to zero at the endpoint. Therefore, $r=r_B$ throughout $(0,\theta_0)$. Then it follows that \eqref{cap:volumeode} is equality almost everywhere. Consequently, at almost every regular level, the entire chain in the following becomes equality
\[
 Y\geq\frac{\mathcal P^2}{X}
       \geq F_t(\mathcal P)\geq F_t(P_0)
\]
with $F_{t}$ introduced in \eqref{eq:Ft_def}. Moreover, the strict increase of $F_t$ forces $\mathcal P=P_0$, or equivalently $\mathcal P(t)^2=4\pi\mathcal V(t)$. This, together with planar isoperimetry, implies $E_t$ is a single disk. Now we use this fact to show that $\Omega$ is indeed a disk by an approximating argument. Choose a decreasing sequence of such regular levels $t_j\to c^{+}$ and write $E_{t_j}=B(z_j,R_j)$, a family of nested disks, whose radius increases and tends to $R:=\sqrt{A/\pi}$. For $j>i$, the nesting relation gives $|z_j-z_i|\leq R_j-R_i$; hence $z_i$ converges to some $z_*$. This, together with the estimate $|z_*-z_i|\leq R-R_i$, implies that $E_{t_i}$ is contained in $B(z_*,R)$ for sufficiently large $i$. Conversely, every point of $B(z_*,R)$ belongs to $E_{t_i}$ for sufficiently large $i$ by the convergence of centers and radius. In conclusion, the limiting disk $B(z_*,R)$ is contained in $\Omega$ but has the same area as $\Omega$, leading to the desired conclusion $\Omega=B(z_*,R)$.


\bigskip
\textbf{Conflict of Interest Statement}: On behalf of all authors, the
corresponding author states that there is no conflict of interest.

\bigskip
\textbf{Data Availability Statement}: On behalf of all authors, the
corresponding author confirms that the current paper does not have
associated data.

\bigskip
\textbf{Declaration of AI Usage}: The authors first realized that the first positive eigenspinor should possess a non-vanishing upper component (after checking a series of numerical tests produced by ChatGPT 6 Astra), which initiates the attempts for the proof and the idea that a deformation argument might be available. After transferring these observations to ChatGPT 6 Astra, it provides a detailed proof, including the verification of non-vanishing upper components (Theorem \ref{con:nonvanishing}) and the geometric comparison bridge (Theorem \ref{cap:sharp}). The authors then substantially simplified the arguments and rewrote the manuscript.

\nolinenumbers
\phantomsection

\end{document}